\documentclass[aop]{imsart}

\RequirePackage{amsthm,amsmath,amsfonts,amssymb}
\RequirePackage[numbers,sort&compress]{natbib}
\RequirePackage[colorlinks,citecolor=blue,urlcolor=blue]{hyperref}
\RequirePackage{graphicx}

\startlocaldefs
\theoremstyle{plain}
\newtheorem{axiom}{Axiom}
\newtheorem{claim}[axiom]{Claim}
\newtheorem{theorem}{Theorem}[section]
\newtheorem{lemma}[theorem]{Lemma}
\newtheorem{proposition}{Proposition}
\newtheorem{corollary}[theorem]{Corollary}
\theoremstyle{definition}

 \newtheorem{remark}{Remark}

\endlocaldefs

\begin{document}

\begin{frontmatter}
\title{Percolation of both signs in a triangular-type 3D Ising model above $T_c$}
\runtitle{Percolation of both signs in Ising model}

\begin{aug}
\author[A]{\fnms{Jianping}~\snm{Jiang}\ead[label=e1]{jianpingjiang@tsinghua.edu.cn}},
\author[B]{\fnms{Sike}~\snm{Lang}\ead[label=e2]{langsk24@mails.tsinghua.edu.cn}}
\address[A]{Yau Mathematical Sciences Center, Tsinghua University, Beijing 100084, China. \printead[presep={ ,\ }]{e1}}

\address[B]{Qiuzhen College, Tsinghua University, Beijing 100084, China. \printead[presep={,\ }]{e2}}
\end{aug}

\begin{abstract}
Let $\mathbb{T}$ be the two-dimensional triangular lattice, and $\mathbb{Z}$ be the one-dimensional integer lattice. Let $\mathbb{T}\times \mathbb{Z}$ denote the Cartesian product graph. Consider the Ising model defined on this graph with inverse temperature $\beta$ and external field $h$, and let $\beta_c$ be the critical inverse temperature when $h=0$. We prove that for each $\beta\in[0,\beta_c)$, there exists $h_c(\beta)>0$ such that both a unique infinite $+$cluster and a unique infinite $-$cluster coexist whenever $|h|<h_c(\beta)$. The same coexistence result also holds for the three-dimensional triangular lattice.
\end{abstract}



\end{frontmatter}

\section{Introduction and main result}
For the two-dimensional Ising model on the square lattice $\mathbb{Z}^2$, it was proved by Higuchi \cite{Hig82,Hig93} that there is coexistence of infinite  $+$ star-cluster and infinite $-$ star-clusters ("star" means adding all diagonals of $\mathbb{Z}^2$) when the inverse temperature $\beta\in[0,\beta_c(\mathbb{Z}^2))$ and there is no external field; but neither an infinite $+$cluster nor an infinite $-$cluster exists in the same regime. For the Ising model on the two-dimensional triangular lattice $\mathbb{T}$, it is known (see \cite{GS09} and \cite{BCM10}) that there is  neither an infinite $+$cluster nor an infinite $-$cluster whenever $\beta\in[0,\beta_c(\mathbb{T}))$ and the external field $h=0$.

For the three (and above) dimensional Ising model, the behavior of infinite $+/-$ clusters are strikingly different from the two-dimensional case. It was proved by Campanino and Russo \cite{CR85} that for the Ising model on the cubic lattice $\mathbb{Z}^3$,  there is coexistence of infinite $+$cluster and infinite $-$cluster when $\beta$ is small (with positive distance away from the critical inverse temperature $\beta_c(\mathbb{Z}^3)$) and $|h|$ is small. It is expected that this coexistence result should persist up to $\beta_c(\mathbb{Z}^3)$ and perhaps surprisingly even beyond $\beta_c(\mathbb{Z}^3)$ (see, e.g., \cite{ABL87}). Actually, it is believed that the coexistence of infinite clusters of both signs for $\beta$ slightly larger than $\beta_c(\mathbb{Z}^3)$ and $h=0$ should be useful to understand the roughening transition in 3D \cite{BF82,ACCFR83}. In this paper, we study the coexistence of infinite clusters of both signs in 3D Ising models and we will prove the coexistence holds up to the critical inverse temperature for certain 3D lattices.  When the dimension $d$ is large enough, it was proved by Aizenman, Bricmont and Lebowitz  \cite{ABL87} that there is an infinite $+$cluster in the minus phase (i.e., the infinite-volume measure obtained from all $-$ boundary conditions) for each $\beta\in [0,\beta_+]$ and $h$ larger than some negative number where $\beta_+>\beta_c(\mathbb{Z}^d)$. 

Let $\mathbb{Z}$ be the set of all integers. With a slight abuse of notation, we also use $\mathbb{Z}$ to denote the one-dimensional integer lattice, i.e., with the set of vertices $\mathbb{Z}$ and the set of edges $\{xy: x, y\in\mathbb{Z}, |x-y|=1\}$. Note that $\mathbb{T}$ can be obtained from the square lattice $\mathbb{Z}^2$ by adding one of the diagonals; see Figure \ref{fig:T}. Let $\mathbb{T} \times \mathbb{Z}$ be the Cartesian product of the graphs $\mathbb{T}$ and $\mathbb{Z}$. We consider the Ising model defined on $\mathbb{T} \times \mathbb{Z}$. It is well-known since Peierls \cite{Pei36} that when $h=0$, this model exhibits a phase transition at some critical inverse temperature $\beta_c(\mathbb{T} \times \mathbb{Z}) \in (0,\infty)$. Actually, we also know that the phase transition is sharp \cite{ABF87,DCT16,DCRT19}. 

\begin{figure}
	\begin{center}
		\includegraphics{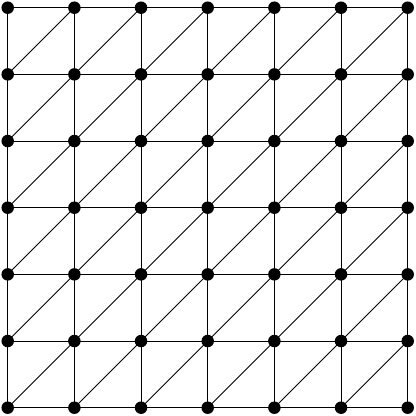}
		\caption{The triangular lattice can be obtained from the square lattice by adding one of the diagonals.}\label{fig:T}
	\end{center}
\end{figure}

The set of vertices for the graph $\mathbb{T} \times \mathbb{Z}$ is $\mathbb{Z}^3$. Let $\Lambda\subset \mathbb{Z}^3$ be a finite subset.  We can (and will) view $\Lambda$ as a subgraph of $\mathbb{T} \times \mathbb{Z}$  with the set of vertices $\Lambda$ and the set of edges $\{xy: xy \text{ is an edge in }\mathbb{T} \times \mathbb{Z}, x,y\in \Lambda\}$. The Ising model on the graph $\Lambda$ at inverse temperature $\beta\geq 0$ with boundary conditions $\eta \in \{-1, 0, +1\}^{\mathbb{Z}^3\setminus \Lambda}$ and external field $\mathbf{h}\in{\mathbb{R}}^{\Lambda}$ is the probability measure $P_{\Lambda,\beta,\mathbf{h}}$ on $\{-1,+1\}^{\Lambda}$ such that
\begin{equation}
	P_{\Lambda,\beta,\mathbf{h}}^{\eta}(\sigma):=\frac{\exp[\beta\sum_{xy\in\Lambda}\sigma_x\sigma_y+\beta\sum_{xy:x\in\Lambda,y\notin \Lambda}\sigma_x\eta_y+\sum_{x\in\Lambda}h_x\sigma_x]}{Z_{\Lambda,\beta,\mathbf{h}}^{\eta}}, ~\forall \sigma \in\{-1,+1\}^{\Lambda},
\end{equation}
where the first sum is over all edges in $\mathbb{T}\times\mathbb{Z}$ with both endpoints in $\Lambda$, the second sum is over all edges in $\mathbb{T}\times\mathbb{Z}$ with exactly one endpoint in $\Lambda$ and the third sum is over all vertices in $\Lambda$, and $Z_{\Lambda,\beta,\mathbf{h}}^{\eta}$ is the partition function. We write $P_{\Lambda,\beta,h}^{\eta}$ if $h_x=h$ for each $x\in \Lambda$, and $P_{\Lambda,\beta,h}$ (respectively, $P_{\Lambda,\beta,h}^{\pm}$) if in addition $\eta_x=0$ (respectively, $\eta_x=\pm$) for each $x\in\mathbb{Z}^3\setminus\Lambda$; they correspond to the free (all plus and all minus) boundary conditions. By the GKS inequalities \cite{Gri67,KS68}, as $\Lambda\uparrow\mathbb{T} \times \mathbb{Z}$, $P_{\Lambda,\beta,h}$ converges weakly to an infinite-volume measure, which we denote by $P_{\beta, h}$. When $\beta\in[0,\beta_c(\mathbb{T} \times \mathbb{Z}))$ and $h\in\mathbb{R}$, the exponential decay result of  \cite{ABF87,DCT16,DCRT19} at $h=0$ and monotonicity of truncated two-point functions in $h$ \cite{GHS70, DSS23} imply that there is a unique infinite-volume limit regardless of the boundary conditions (i.e., both $P_{\Lambda,\beta,h}^+$ and $P_{\Lambda,\beta,h}^-$ converge to the same limit). For a fixed configuration $\sigma\in\{-1,+1\}^{\mathbb{Z}^3}$, we say that there is an infinite $+$cluster if the subgraph induced in $\mathbb{T}\times\mathbb{Z}$ by $\{\sigma\in\{-1,+1\}^{\mathbb{Z}^3}:\sigma_x=+1\}$ contains a connected component with infinitely many vertices; the definition of an infinite $-$cluster is similar. Our main result is
\begin{theorem}\label{thm1}
	Consider the Ising model defined on $\mathbb{T} \times \mathbb{Z}$ with inverse temperature $\beta$ and  external field $h$. For each $\beta\in[0,\beta_c(\mathbb{T} \times \mathbb{Z}))$, there exists $h_c(\beta)>0$ such that whenever $|h|<h_c(\beta)$, there exist both a unique infinite $+$cluster and a unique infinite $-$cluster $P_{\beta,h}$ almost surely.
\end{theorem}

\begin{remark}
	In the proof, we will see that we may define $h_c(\beta)$ in such a way that there is no infinite $+$cluster if $h<-h_c(\beta)$ and there is no infinite $-$cluster if $h>h_c(\beta)$ $P_{\beta,h}$ almost surely. We believe but cannot prove at the moment that there is no infinite $+$cluster at $-h_c(\beta)$ and no  infinite $-$cluster at $h_c(\beta)$.
\end{remark}

\begin{remark}
	Our proof strategy of Theorem \ref{thm1} applies to more general graphs, like $\mathbb{T} \times G$ where $G$ is a graph that contains more than one edge. In particular, Theorem \ref{thm1} holds for $\mathbb{T} \times K_2$ where $K_2$ is the complete graph with two vertices and also for the three-dimensional triangular lattice.
\end{remark}

\begin{remark}
	As mentioned above, Theorem \ref{thm1} is also expected to hold for the cubic lattice $\mathbb{Z}^3$. The arguments in Sections \ref{sec:com} and \ref{sec:diff} work well for the cubic lattice. The missing ingredient is that the $+$ connection (and thus also $-$ connection) in some slab $\mathbb{Z}^2\times\{0,\dots,k\}$ is at least critical at $\beta\in[0,\beta_c(\mathbb{Z}^3))$ and $h=0$.
\end{remark}

\begin{remark}
	On $\mathbb{Z}^2$, Higuchi's coexistence result \cite{Hig82,Hig93} is somewhat artificial in the sense that the extra bonds (i.e., the diagonals of $\mathbb{Z}^2$) have no Ising interactions. In the context of the current paper, a natural question is what happens for the Ising model on  $\mathbb{Z}^2$ with both nearest-neighbor and diagonal interactions (we denote this lattice by  $\mathbb{Z}^2_*$, which is the matching graph of the usual nearest-neighbor square lattice). From \cite{Hig82,Hig93}, it is easy to show that there is coexistence of infinite $+$ and $-$clusters for the Ising model defined on $\mathbb{Z}^2_*$ when $\beta$ is small. Then it is natural to conjecture that the same coexistence result holds for all $\beta\in[0,\beta_c(\mathbb{Z}^2_*))$. Our method (see a summary in the next paragraph) should be applicable to this case: $\mathbb{Z}^2_*$ can be viewed as $\mathbb{T}$ plus diagonal edges with slope $-1$ (see Figure \ref{fig:T}) and each of such edges is selected with probability $p$ for an increasing event.
\end{remark}
The main idea of the proof of Theorem \ref{thm1} is by introducing a new parameter, $p$, which is the probability that a given vertex in $\mathbb{Z}^3$ can be used for an increasing event $A$. So the probability of $A$, which we denote by $\mu_{\Lambda,p,h}(A)$, is a function of $p$ and $h$ (we always fix $\beta<\beta_c(\mathbb{T} \times \mathbb{Z})$). Most of the paper is devoted to proving that the ratio of the partial derivatives of $\mu_{\Lambda,p,h}(A)$ with respect to $h$ and $p$ is bounded from above by a continuous and positive function and this bound is uniform in $\Lambda$. A principal novelty of the current paper is to bound the influence from one fixed spin on the increasing event $A$ by adding an extra exponentially small (in radius) external field on the sphere center at this point (see Proposition \ref{prop:sd} below), and then use this to bound the difference between $h$ dependence at the point and $p$ dependence nearby (i.e., pivotality) by an exponentially small fraction of the sum of $h$ dependence nearby (see Lemma \ref{lem:covtopiv} below).

The method of differential inequalities has been used to prove strict monotonicity for critical points in percolation and Ising models \cite{AG91,BBR14}, strict inequality for critical values of random cluster and Potts models \cite{BGK93}, equalities between critical exponents and Lipschitz continuity of critical lines \cite{Gri95}, inequality between the critical points of site and bond percolation \cite{CS00}, the location and strict monotonicity of the Kert\'{e}sz line \cite{CJN18,HK23}.

The organization of the paper is as follows. In Section \ref{sec:com}, we prove a stochastic inequality which relates the influence on increasing events from one fixed spin and the influence of an external field imposed a certain distance away, and then quantify the latter influence.  In Section \ref{sec:diff}, we prove an inequality on the partial derivatives $\partial \mu_{\Lambda,p,h}(A)/\partial h$ and $\partial \mu_{\Lambda,p,h}(A)/\partial p$. In Section \ref{sec:pf}, we prove our main result Theorem \ref{thm1}. In the \hyperref[appn]{Appendix}, we prove a pointwise reversed inequality on $h$ and $p$ dependence.

\section{Some comparison and correlation inequalities}\label{sec:com}
In this section, we prove some comparison and correlation inequalities which will be useful later. 

For $x,y\in\mathbb{Z}^3$, we use $|x-y|$ to denote the Euclidean distance between $x$ and $y$. Let $\mathbb{N}$ be the set of all positive integers. For $n\in\mathbb{N}$, let $B_n=[-n,n]^3\cap (\mathbb{T}\times \mathbb{Z})$ be the box with side-length $2n$ centered at the origin. For $x\in\mathbb{Z}^3$, let $B_n(x):=x+B_n$ be the translation of $B_n$ by $x$. For $\Lambda \subset \mathbb{Z}^3$,  let $|\Lambda|:=\text{number of vertices in }\Lambda$, $\Lambda^c:=\mathbb{Z}^3\setminus \Lambda$ and
\[\partial \Lambda:=\{x\in\Lambda: \exists y\in \Lambda^c \text{ such that } xy \text{ is an edge in }\mathbb{T}\times \mathbb{Z}\};\]
let $\mathcal{F}_{\Lambda}$ denote the $\sigma$-algebra generated by all events which are determined on $\Lambda$.

For $A\in \mathcal{F}_{\Lambda}$, we write $I_A$ for the indicator function of $A$. We sometimes write $\langle \cdot \rangle_{\Lambda,\beta,\mathbf{h}}$ for the expectation with respect to $P_{\Lambda,\beta,\mathbf{h}}$, and $\langle X; Y \rangle_{\Lambda,\beta,\mathbf{h}}$ for the corresponding covariance between two random variables $X$ and $Y$.

We first recall a ratio mixing property for high temperature Ising models.
\begin{proposition}[\cite{Ale04,DSS23}]\label{prop:ratio}
	For each $\beta\in[0,\beta_c(\mathbb{T}\times \mathbb{Z}))$, there exist constants $K=K(\beta)\in(0,\infty)$ and $\lambda=\lambda(\beta)\in(0,\infty)$ such that for each finite $\Lambda \subset \mathbb{Z}^3$ and all $\Lambda_1, \Lambda_2\subset \Lambda$ with $\Lambda_1\cap\Lambda_2=\emptyset$, we have
	\[\left|\frac{P_{\Lambda,\beta,h}(A \cap B)}{P_{\Lambda,\beta,h}(A)P_{\Lambda,\beta,h}(B)}-1\right|\leq K\sum_{x \in \Lambda_1, y\in \Lambda_2}e^{-\lambda|x-y|},~\forall A\in\mathcal{F}_{\Lambda_1}, B\in\mathcal{F}_{\Lambda_2},\]
	whenever the right hand side (RHS) of the inequality is less than $1$.
\end{proposition}
\begin{proof}
	This follows from Theorem 3.10 and Remark 3.11 of \cite{Ale04} (the actual proof is in Section 5 of \cite{Ale98}) and the strong spatial mixing property for $P_{\Lambda,\beta,h}$ whenever $\beta\in [0,\beta_c(\mathbb{T}\times \mathbb{Z}))$ and $h\in\mathbb{R}$ from Corollary 1.8 of \cite{DSS23}. We remark that as in \cite{DSS23}, we may choose the decay rate $\lambda$ to depend only on $\beta$ but not on $h$.
\end{proof}

We next use the ratio mixing property to prove the following stochastic comparison inequality.
\begin{proposition}\label{prop:sd}
	For each $\beta\in[0,\beta_c(\mathbb{T}\times \mathbb{Z}))$ and $h\in \mathbb{R}$, there exists $a=a(\beta)\in(0,\infty)$ and  $M=M(\beta,h)\in(0,\infty)$ such that for each finite $\Lambda\subset \mathbb{Z}^3$, we have that for any $x\in \Lambda$, any $m\geq M$ and any increasing event $A\in\mathcal{F}_{\Lambda\setminus B_m(x)}$,
	\[P_{\Lambda,\beta,\mathbf{g}^+}(A)\geq P_{\Lambda,\beta,h}(A | \sigma_x=+1), P_{\Lambda,\beta,\mathbf{g}^-}(A)\leq P_{\Lambda,\beta,h}(A | \sigma_x=-1),\]
	where
	\[g_y^{\pm}:=\begin{cases}h \pm e^{-am}, &y\in \partial B_m(x)\cap\Lambda,\\
		h, &y\in \Lambda \setminus \partial B_m(x).
	\end{cases}\]
\end{proposition}
\begin{proof}
	We order all vertices in $\partial B_m(x)\cap \Lambda$ as $x_1, x_2, \dots, x_L$ for some $L\in\mathbb{N}$. We construct algorithmically a coupling for spins from both measures by exploring the statuses of $\sigma\overset{d} = P_{\Lambda,\beta,\mathbf{g}^+}$ and $\tau\overset{d}=P_{\Lambda,\beta,h}$ in $\partial B_m(x)\cap\Lambda$ one by one according to the fixed order. The first inequality in the proposition follows from the following claim, the spatial Markov property and monotonicity in boundary conditions for Ising measures; the second inequality is proved similarly.
	\begin{claim} For each integer $k\in [0,L-1]$, any $\tilde{\sigma},\tilde{\tau}\in\{-1,+1\}^{k}$ with $\tilde{\sigma}_i\geq \tilde{\tau}_i$ for each $i$, we have
		\begin{align*}
			&P_{\Lambda,\beta,\mathbf{g}^+}(\sigma_{x_{k+1}}=+1|\sigma_{x_1}=\tilde{\sigma}_1,\dots, \sigma_{x_k}=\tilde{\sigma}_k)\\
			&\qquad \geq P_{\Lambda,\beta,h}(\tau_{x_{k+1}}=+1|\tau_{x_1}=\tilde{\tau}_1,\dots, \tau_{x_k}=\tilde{\tau}_k, \tau_x=+1).
		\end{align*}
	\end{claim}
	\begin{proof}[Proof of Claim]
		Note that 
		\[P_{\Lambda,\beta,\mathbf{g}^+}(\cdot|\sigma_{x_1}=\tilde{\sigma}_1,\dots, \sigma_{x_k}=\tilde{\sigma}_k)=P_{\Lambda,\beta,\mathbf{g}^+}^{\tilde{\sigma}[k]}(\cdot), \text{ with } \tilde{\sigma}[k]:=\{\sigma_{x_1}=\tilde{\sigma}_1,\dots, \sigma_{x_k}=\tilde{\sigma}_k\}.\]
		We define $\mathbf{g}(t)\in\mathbb{R}^{\Lambda}$ by
		\begin{equation}\label{eq:g(t)}
			g_y(t):=\begin{cases}h + t, &y\in \partial B_m(x)\cap\Lambda,\\
				h, &y\in \Lambda \setminus \partial B_m(x).
			\end{cases}
		\end{equation}
		Then there exists a $c_1=c_1(\beta,h)\in(0,\infty)$ such that
		\begin{align*}
			\frac{\partial P_{\Lambda,\beta,\mathbf{g}(t)}^{\tilde{\sigma}[k]}(\sigma_{x_{k+1}}=+1)}{\partial t}&=\sum_{y\in \partial B_m(x)\cap\Lambda} \left. \frac{\partial P_{\Lambda,\beta,\mathbf{g}}^{\tilde{\sigma}[k]}(\sigma_{x_{k+1}}=+1)}{\partial g_y} \right|_{\mathbf{g}=\mathbf{g}(t)}\\
			&=\sum_{y\in \partial B_m(x)\cap \Lambda}\frac{\langle \sigma_{x_{k+1}};\sigma_y\rangle_{\Lambda,\beta,\mathbf{g}(t)}^{\tilde{\sigma}[k]}}{2}\geq \frac{\langle \sigma_{x_{k+1}};\sigma_{x_{k+1}}\rangle_{\Lambda,\beta,\mathbf{g}(t)}^{\tilde{\sigma}[k]}}{2}\geq c_1,
		\end{align*}
		where we have used the FKG inequality in the first inequality.
		
		Therefore,
		\begin{align*}
			P_{\Lambda,\beta,\mathbf{g}^+}^{\tilde{\sigma}[k]}(\sigma_{x_{k+1}}=+1)&=\int_0^{e^{-am}} \frac{\partial P_{\Lambda,\beta,\mathbf{g}(t)}^{\tilde{\sigma}[k]}(\sigma_{x_{k+1}}=+1)}{\partial t} dt+P_{\Lambda,\beta,h}^{\tilde{\sigma}[k]}(\sigma_{x_{k+1}}=+1)\\
			&\geq c_1 e^{-am}+P_{\Lambda,\beta,h}^{\tilde{\sigma}[k]}(\sigma_{x_{k+1}}=+1).
		\end{align*}
		Meanwhile, we also have
		\begin{align*}
			&P_{\Lambda,\beta,h}(\tau_{x_{k+1}}=+1|\tau_{x_1}=\tilde{\tau}_1,\dots, \tau_{x_k}=\tilde{\tau}_k, \tau_x=+1)\\
			&\qquad=\frac{P_{\Lambda,\beta,h}(\tau_{x_{k+1}}=+1, \tau_{x_1}=\tilde{\tau}_1,\dots, \tau_{x_k}=\tilde{\tau}_k | \tau_x=+1)}{P_{\Lambda,\beta,h}(\tau_{x_1}=\tilde{\tau}_1,\dots, \tau_{x_k}=\tilde{\tau}_k | \tau_x=+1)}.
		\end{align*}
		Applying Proposition \ref{prop:ratio} to the last displayed fraction, we get
		\begin{align*}
			&P_{\Lambda,\beta,h}(\tau_{x_{k+1}}=+1|\tau_{x_1}=\tilde{\tau}_1,\dots, \tau_{x_k}=\tilde{\tau}_k, \tau_x=+1)\\
			&\qquad \leq \frac{(1+K|\partial B_m|e^{-\lambda m})P_{\Lambda,\beta,h}(\tau_{x_{k+1}}=+1, \tau_{x_1}=\tilde{\tau}_1,\dots, \tau_{x_k}=\tilde{\tau}_k)}{(1-K|\partial B_m|e^{-\lambda m})P_{\Lambda,\beta,h}( \tau_{x_1}=\tilde{\tau}_1,\dots, \tau_{x_k}=\tilde{\tau}_k)}\\
			&\qquad=\frac{(1+K|\partial B_m|e^{-\lambda m}))}{(1-K|\partial B_m|e^{-\lambda m})}P_{\Lambda,\beta,h}^{\tilde{\tau}[k]}(\tau_{x_{k+1}}=+1)\\
			&\qquad \leq P_{\Lambda,\beta,h}^{\tilde{\tau}[k]}(\tau_{x_{k+1}}=+1)+4K|\partial B_m|e^{-\lambda m}, ~\forall m\geq M_1
		\end{align*}
		if  $M_1>0$ is chosen such that
		\[K|\partial B_m|e^{-\lambda m}<1/2,~\forall m\geq M_1.\]
		We now pick any $a\in(0,\lambda)$ and $M \geq M_1$ such that
		\[c_1 e^{-am} \geq 4K|\partial B_m|e^{-\lambda m},~\forall m\geq M.\]
		Then the claim follows by induction on $k$ (note that the above proof already includes the base case $k=0$).
	\end{proof}
	As mentioned before the claim, this also completes the proof of the proposition.
\end{proof}

Next, we quantify the influence on increasing events from the measure with an extra field imposed on $\partial B_m(x)\cap \Lambda$ by the influence from the measure without.
\begin{proposition}\label{prop:cov}
	For any finite $\Lambda\subset \mathbb{Z}^3$ and $x\in \Lambda$, let $\mathbf{g}(t)$ be defined by \eqref{eq:g(t)} for some $m\in\mathbb{N}$. Then we have that for each increasing event $A\in\mathcal{F}_{\Lambda\setminus B_m(x)}$,
	\begin{align*}
		\exp\left[-4 |\partial B_m| |t|\right] \sum_{y\in \partial B_m(x)\cap \Lambda}\langle \sigma_y; I_A\rangle_{\Lambda,\beta,h} &\leq \sum_{y\in \partial B_m(x)\cap \Lambda}\langle \sigma_y; I_A\rangle_{\Lambda,\beta,\mathbf{g}(t)}\\
		&\qquad \leq \exp\left[4 |\partial B_m| |t|\right] \sum_{y\in \partial B_m(x)\cap \Lambda}\langle \sigma_y; I_A\rangle_{\Lambda,\beta,h}.
	\end{align*}
\end{proposition} 
The proof of the proposition uses the following lemma, which is of independent interest.
\begin{lemma}\label{lem:betaandh}
	For any finite $\Lambda\subset \mathbb{Z}^3$ and any $\mathbf{h}\in\mathbb{R}^{\Lambda}$, we have that for any $y, z\in\Lambda$ and any increasing event $A\in\mathcal{F}_{\Lambda}$
	\[\left| \langle \sigma_y\sigma_z; I_A\rangle_{\Lambda,\beta,\mathbf{h}} \right| \leq \langle \sigma_y; I_A\rangle_{\Lambda,\beta,\mathbf{h}} +\langle \sigma_z; I_A\rangle_{\Lambda,\beta,\mathbf{h}}.\]
\end{lemma}
\begin{proof}
	This follows from the observation that $\sigma_x+\sigma_y\pm \sigma_x\sigma_y$ are increasing functions and the FKG inequality. We note that the covariance on the left-hand side can in fact be negative. E.g., this occurs if $A:=\{\sigma_y=+1\}$ and $\langle \sigma_y \rangle_{\Lambda,\beta,\mathbf{h}}=\langle \sigma_z \rangle_{\Lambda,\beta,\mathbf{h}} \in(-1,0)$.
\end{proof}
\begin{proof}[Proof of Proposition \ref{prop:cov}]
	Let
	\[f(t):=\sum_{y\in \partial B_m(x)\cap \Lambda}\langle \sigma_y; I_A\rangle_{\Lambda,\beta,\mathbf{g}(t)}.\]
	Then we have 
	\begin{align*}
		\frac{d f(t)}{d t}&=\sum_{y, z\in \partial B_m(x)\cap \Lambda}\langle \sigma_y\sigma_z; I_A\rangle_{\Lambda,\beta,\mathbf{g}(t)}-\langle \sigma_y\rangle_{\Lambda,\beta,\mathbf{g}(t)}\langle \sigma_z; I_A\rangle_{\Lambda,\beta,\mathbf{g}(t)}\\
		&\qquad-\langle \sigma_z\rangle_{\Lambda,\beta,\mathbf{g}(t)}\langle \sigma_y; I_A\rangle_{\Lambda,\beta,\mathbf{g}(t)}.
	\end{align*}
	Lemma \ref{lem:betaandh} and the trivial bound $|\langle \sigma_y\rangle_{\Lambda,\beta,\mathbf{g}(t)}|\leq 1$ imply that
	\[\left| \frac{d f(t)}{d t} \right| \leq \sum_{y, z\in \partial B_m(x)\cap \Lambda}\left[2\langle \sigma_y; I_A\rangle_{\Lambda,\beta,\mathbf{g}(t)}+2\langle \sigma_z; I_A\rangle_{\Lambda,\beta,\mathbf{g}(t)}\right] \leq 4|\partial B_m| f(t).\]
	The proposition follows from Gr\"onwall's inequality.
\end{proof}

\section{A differential inequality}\label{sec:diff}
Let $\Lambda\subset \mathbb{Z}^3$ be finite. For any $A\in\mathcal{F}_{\Lambda}$, $\Delta\subset \Lambda$ and $\sigma\in\{-1,+1\}^{\Lambda}$, we say {\it $\Delta$ is pivotal for $(A,\sigma)$} if $\sigma^{\Delta +}\in A$ and $\sigma^{\Delta -}\notin A$ where
\begin{equation}\label{eq:sigmamod}
	\sigma_x^{\Delta+}:=\begin{cases}
		\sigma_x, & x\in \Lambda\setminus \Delta,\\
		+1,       &  x\in \Delta,
	\end{cases}\qquad
	\sigma_x^{\Delta-}:=\begin{cases}
		\sigma_x, & x\in \Lambda\setminus \Delta,\\
		-1,       &  x\in \Delta.
	\end{cases}
\end{equation}
We define the event
\[\{\Delta \text{ is pivotal for }A\}:=\{\sigma\in\{-1,+1\}^{\Lambda}: \Delta \text{ is pivotal for }(A,\sigma)\}.\]
We also define $\{A \text{ occurs on } \Delta \}$ to be the event 
\[\left\{\sigma\in \{-1,+1\}^{\Lambda}: \forall \tau \in  \{-1,+1\}^{\Lambda} \text{ with }\tau_x=\sigma_x \text{ for each } x\in \Delta \text{ implies } \tau \in A\right\}.\]
Note that for any $A\in\mathcal{F}_{\Lambda}$, we have
\begin{equation}\label{eq:pivcom}
	A \cap \{\Delta \text{ is pivotal for } A\}^c=\{A \text{ occurs on }  \Delta^c\}.
\end{equation}

For $k\in \mathbb{N}\cup \{0\}$, let $\mathbb{S}_k$ be the slab of thickness $k$, namely the induced subgraph of $\mathbb{T}\times \mathbb{Z}$ with the set of vertices $\mathbb{Z}^2\times\{0,\dots,k\}$. For the rest of the section, we always fix $\beta\in[0,\beta_c(\mathbb{T}\times\mathbb{Z}))$ and will drop the $\beta$ dependence. For any finite $\Lambda\subset\mathbb{Z}^3$ and $p\in[0,1]$, we define a new measure $\mu_{\Lambda,p,h}$ in the following way: for each $A\in \mathcal{F}_{\Lambda\cap \mathbb{S}^1}$,
\[\mu_{\Lambda,p,h}(A):=\sum_{\omega\in \{0,1\}^{\Lambda_{(1)}}}p^{|\omega^{-1}(1)|}(1-p)^{|\omega^{-1}(0)|}P_{\Lambda,\beta,h}(A \text{ occurs on } \mathbb{S}_0 \cup \omega^{-1}(1)),\]
where $\Lambda_{(1)}:=\{x\in \Lambda: \text{ the third coordinate of } x, x(3)=1\}$, $\omega^{-1}(1):=\{x\in \Lambda_{(1)}: \omega_x=1\}$ and $\omega^{-1}(0):=\Lambda_{(1)}\setminus \omega^{-1}(1)$. That is, we first run an independent Bernoulli site percolation on $\mathbb{S}_1\setminus \mathbb{S}_0$ to get a configuration $\omega$, and $\mu_{\Lambda,p,h}(A)$ is probability that $A$ occurs on $\mathbb{S}_0\cup \omega^{-1}(1)$.

Our main ingredient to the proof of Theorem \ref{thm1} is the following differential inequality. We define the event $\{\exists~\text{LR} +\text{crossing in } B_n\cap \mathbb{S}_1\}$ to be
\[\{\exists \text{ a path of }+ \text{ spins in }  B_n\cap \mathbb{S}_1 \text{ which intersects both hyperplanes }\{x(1)=\pm n\}\}.\]
\begin{proposition}\label{prop:diff}
	There exists a continuous function $\gamma:(0,1)\times \mathbb{R}\rightarrow(0,\infty)$ such that for any finite $\Lambda\subset\mathbb{Z}^3$ with $B_n\subset \Lambda$ for some $n\in\mathbb{N}$, the increasing event $\mathcal{H}:=\{\exists \text{ LR } + \text{crossing in }B_n\cap \mathbb{S}_1 \}$ satisfies
	\[\frac{\partial \mu_{\Lambda,p,h}(\mathcal{H})}{\partial h}\leq \gamma(p,h) \frac{\partial \mu_{\Lambda,p,h}(\mathcal{H})}{\partial p}.\]
\end{proposition}
\begin{remark}
	Proposition \ref{prop:diff} holds for more general increasing events in $\mathcal{F}_{\Lambda\cap\mathbb{S}_1}\setminus  \mathcal{F}_{\Lambda\cap\mathbb{S}_0}$. For example, it holds for the event $\{\exists \text{ a path of }+ \text{spins in }\mathbb{S}_1 \text{ which connects }0 \text{ to }\partial B_n\}$. But it certainly fails if $A\in \mathcal{F}_{\Lambda\cap\mathbb{S}_0}$ with $A$ nontrivial (i.e., $\mu_{\Lambda,p,h}(A)$ is not a constant) since the RHS of the inequality is always $0$ in this case.
\end{remark}

We first compute both partial derivatives in Proposition \ref{prop:diff}. For $x\in \Lambda_{(1)}$, $\Delta\subset \Lambda_{(1)}$, $\omega\in\{0,1\}^{\Delta}$, we define (the first two equations generalize our previous definition of $\omega^{-1}(1)$ and $\omega^{-1}(0)$)
\begin{align*}
	&\omega^{-1}(1):=\{y\in \Delta: \omega_y=1\}, \omega^{-1}(0):=\Delta\setminus \omega^{-1}(1),\\
	&\mathcal{H}(x,\omega):=\{\mathcal{H} \text{ occurs on } \mathbb{S}_0 \cup \omega^{-1}(1)\cup \{x\}\},\\
	&\{x \text{ is}+\text{pivotal for } \mathcal{H}(x,\omega)\}:=\{x \text{ is pivotal for } \mathcal{H}(x,\omega)\}\cap\{\sigma_x=+1\};
\end{align*} 
we remark that it is possible that $x\in\omega^{-1}(1)$.
Note that for each $\omega\in \{0,1\}^{\Lambda_{(1)}\setminus\{x\}}$, we have
\begin{equation}\label{eq:+piv}
	\begin{split}
		&\{x \text{ is}+\text{pivotal for } \mathcal{H}(x,\omega)\}\\
		=&\{\mathcal{H} \text{ occurs on } \mathbb{S}_0 \cup \omega^{-1}(1)\cup \{x\}\}\setminus \{\mathcal{H} \text{ occurs on } \mathbb{S}_0\cup \omega^{-1}(1)\}. 
	\end{split}
\end{equation}
See Figure \ref{fig:Pivotal} for an illustration.

\begin{figure}
	\begin{center}
		\includegraphics{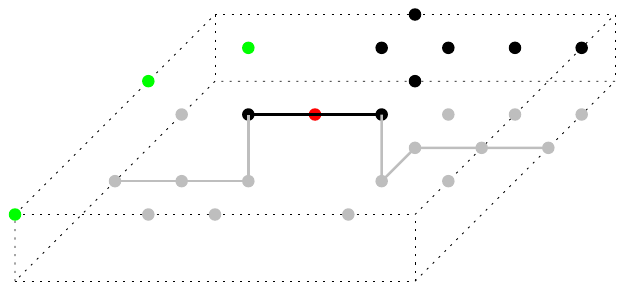}
		\caption{An illustration of the event $\mathcal{H}$ defined in Proposition \ref{prop:diff}. Here $\Lambda=B_{100}$ and $B_n=B_3$; only $B_3\cap\mathbb{S}_1$ is shown (with dotted boundary). A configuration $\sigma\in\{-1,+1\}^{\Lambda}$ is shown; with the set of all plus spins from $\sigma$ consists of all gray, black, red and green points. The gray points are from $\mathbb{S}_0$, and black, red and green points are from $\mathbb{S}_1\setminus\mathbb{S}_0$.  The path made of gray and black line segments crosses $B_3\cap\mathbb{S}_1$ from left to right. Let $x:=(0,0,1)$ be the red point. A configuration $\omega\in \{0,1\}^{\Lambda_{(1)}\setminus\{x\}}$ is shown, with $\omega^{-1}(1)$ consisting of all black points.  So the red and green points take spin value $+1$ but are not included in $\omega^{-1}(1)$. Here $x$ is $+$ pivotal for $(\mathcal{H}(x,\omega),\sigma)$.}\label{fig:Pivotal}
	\end{center}
\end{figure}

\begin{lemma}\label{lem:dirivative}
	Under the assumptions in Proposition \ref{prop:diff}, we have
	\begin{align*}
		\frac{\partial \mu_{\Lambda,p,h}(\mathcal{H})}{\partial p}&=\sum_{x\in \Lambda_{(1)}}\sum_{\omega\in \{0,1\}^{\Lambda_{(1)}\setminus\{x\}}}p^{|\omega^{-1}(1)|}(1-p)^{|\omega^{-1}(0)|} P_{\Lambda,\beta,h}(x \text{ is}+\text{pivotal for } \mathcal{H}(x,\omega)),\\
		\frac{\partial \mu_{\Lambda,p,h}(\mathcal{H})}{\partial h}&=\sum_{x\in \Lambda}\sum_{\omega\in \{0,1\}^{\Lambda_{(1)}}}p^{|\omega^{-1}(1)|}(1-p)^{|\omega^{-1}(0)|}\langle\sigma_x; I_{\{\mathcal{H} \text{ occurs on } \mathbb{S}_0 \cup \omega^{-1}(1)\}}\rangle_{\Lambda,\beta,h}.
	\end{align*}
\end{lemma}
\begin{proof}
	We may generalize $\mu_{\Lambda,p,h}$ to $\mu_{\Lambda,\mathbf{p},\mathbf{h}}$ for $\mathbf{p}\in[0,1]^{\Lambda_{(1)}}$ and $\mathbf{h}\in\mathbb{R}^{\Lambda}$ in the following way
	\[\mu_{\Lambda,\mathbf{p},\mathbf{h}}(\mathcal{H}):=\sum_{\omega\in \{0,1\}^{\Lambda_{(1)}}}\prod_{y\in\omega^{-1}(1)}p_y \prod_{y\in\omega^{-1}(0)}(1-p_y)P_{\Lambda,\beta,\mathbf{h}}(\mathcal{H} \text{ occurs on } \mathbb{S}_0 \cup \omega^{-1}(1)).\]
	Then it is easy to see that for any $x\in\Lambda_{(1)}$,
	\begin{align*}
		&\frac{\partial\mu_{\Lambda,\mathbf{p},\mathbf{h}}(\mathcal{H})}{\partial p_x}=\sum_{\omega\in \{0,1\}^{\Lambda_{(1)}\setminus\{x\}}}\prod_{y\in\omega^{-1}(1)}p_y \prod_{y\in\omega^{-1}(0)}(1-p_y)\times\\
		&\qquad\big[P_{\Lambda,\beta,\mathbf{h}}(\mathcal{H} \text{ occurs on } \mathbb{S}_0 \cup \omega^{-1}(1)\cup \{x\}) - P_{\Lambda,\beta,\mathbf{h}}(\mathcal{H} \text{ occurs on } \mathbb{S}_0\cup \omega^{-1}(1)) \big],
	\end{align*}
	and for any $x\in\Lambda$,
	\[\frac{\partial\mu_{\Lambda,\mathbf{p},\mathbf{h}}(\mathcal{H})}{\partial h_x}=\sum_{\omega\in \{0,1\}^{\Lambda_{(1)}}}\prod_{y\in\omega^{-1}(1)}p_y \prod_{y\in\omega^{-1}(0)}(1-p_y)\langle\sigma_x; I_{\{\mathcal{H} \text{ occurs on } \mathbb{S}_0 \cup \omega^{-1}(1)\}}\rangle_{\Lambda,\beta,\mathbf{h}}.\]
	The lemma follows from the chain rule and \eqref{eq:+piv}.
\end{proof}

Before we dive into the proof of Proposition \ref{prop:diff}, we first present the overall idea. Lemma~\ref{lem:dirivative} tells us that we need to bound the weighted sum of the covariance $\langle \sigma_x; I_{\mathcal{H}(x,\omega)}\rangle_{\Lambda,\beta,h}$ over $x \in \Lambda$ by the weighted sum of $P_{\Lambda,\beta,h}(x \text{ is}+\text{pivotal for } \mathcal{H}(x,\omega))$ over $x\in \Lambda_{(1)}$; here the change of the event in the covariance is fine due to \eqref{eq:+piv}. In Lemma \ref{lem:covtopiv}, we write $\mathcal{H}(x,\omega)$ as the disjoint union of $\mathcal{H}(x,\omega)\cap \{B_m(x) \text{ is pivotal for } \mathcal{H}(x,\omega)\}$ and $\{\mathcal{H}(x,\omega) \text{ occurs on } B_m(x)^c\}$. Then $|\langle \sigma_x; I_{\mathcal{H}(x,\omega)}I_{\{B_m(x) \text{ is pivotal for } \mathcal{H}(x,\omega)\}}\rangle_{\Lambda,\beta,h}|$ is bounded above by a sum of $P_{\Lambda,\beta,h}(y \text{ is}+\text{pivotal for } \mathcal{H}(y,\omega^{(y)}))$ over $y\in B_{m+1}(x)$ where $\omega^{(y)}$ differs from $\omega$ only in $B_{3m}(y)$; so this term has the desired upper bound. The other term $\langle \sigma_x; I_{\{\mathcal{H}(x,\omega) \text{ occurs on } B_m(x)^c\}}\rangle_{\Lambda,\beta,h}$ is the hard one, which requires the full machinery of Section \ref{sec:com}. With the help of Proposition \ref{prop:sd} and \ref{prop:cov}, we can bound it above by a similar good upper bound plus another sum of $\langle \sigma_y; I_{\mathcal{H}(y,\omega)}\rangle_{\Lambda,\beta,h}$ over $y\in \partial B_m(x)$; the crucial point is that this second sum is multiplied by $e^{-am}$ which could beat any polynomial of $m$. Consequently, when we sum $\langle \sigma_x; I_{\mathcal{H}(x,\omega)}\rangle_{\Lambda,\beta,h}$ over $x \in \Lambda$ (to get $\partial \mu_{\Lambda,p,h}(\mathcal{H}) /\partial h$), we obtain an upper bound which is the sum of  $\partial \mu_{\Lambda,p,h}(\mathcal{H}) / \partial p$ and $e^{-am} \times \text{Poly}(m) \times  \frac{\partial \mu_{\Lambda,p,h}(\mathcal{H})}{\partial h}$. By choosing $m$ large so that $e^{-am} \times \text{Poly}(m)<1/2$, we arrive at the desired inequality in Proposition \ref{prop:diff}. In the above discussion, we somehow ignore the $p$-dependence in those sums, that is why we need Lemma \ref{lem:xytoy} to justify this.

The following lemma establishes the connection between the two partial derivatives in Proposition \ref{prop:diff}.
\begin{lemma}\label{lem:covtopiv}
	Under the assumption of Proposition \ref{prop:diff}, for any $x,y\in \Lambda$, let $\omega^{(y)}\in\{0,1\}^{\Lambda_{(1)}}$ be obtained from $\omega$ by opening all vertices in $\Lambda_{(1)}\cap B_{3m}(y)$:
	\[\omega^{(y)}_z:=\begin{cases}
		1,&z\in\Lambda_{(1)}\cap B_{3m}(y),\\
		\omega_z,&z\in \Lambda_{(1)}\setminus B_{3m}(y).
	\end{cases}\] 
	Then there exists constant $M=M(\beta,h)\in(0,\infty)$ such that for each $m\geq M$,
	\begin{align*}
		\langle \sigma_x; I_{\mathcal{H}(x,\omega)}\rangle_{\Lambda,\beta,h}&\leq C(m,\beta,h)\sum_{y\in\Lambda_{(1)}\cap B_{m+1}(x)}P_{\Lambda,\beta,h}(y \text{ is pivotal for } \mathcal{H}(y,\omega^{(y)}))\\
		&\qquad+8e^{-ma}\sum_{y\in\partial B_m(x)\cap \Lambda}\langle \sigma_y; I_{\mathcal{H}(y,\omega)}\rangle_{\Lambda,\beta,h},
	\end{align*}
	where $C(m,\beta,h)\in(0,\infty)$ and $a$ is defined in Proposition \ref{prop:sd}.
\end{lemma}
\begin{remark}
	The reader may wonder whether the local modifications for $\omega$ are necessary. We will comment on this after the inequality \eqref{eq:cov1st}.
\end{remark}
\begin{proof}
	Using \eqref{eq:pivcom}, we can write
	\begin{equation}\label{eq:covdecom}
		\begin{split}
			\langle \sigma_x; I_{\mathcal{H}(x,\omega)}\rangle_{\Lambda,\beta,h}&=\langle \sigma_x; I_{\mathcal{H}(x,\omega)}I_{\{B_m(x) \text{ is pivotal for } \mathcal{H}(x,\omega)\}}\rangle_{\Lambda,\beta,h}\\
			&\qquad+\langle \sigma_x; I_{\mathcal{H}(x,\omega)}I_{\{\mathcal{H}(x,\omega) \text{ occurs on } B_m(x)^c\}}\rangle_{\Lambda,\beta,h}.
		\end{split}
	\end{equation}
	The first term on the RHS of \eqref{eq:covdecom} satisfies that
	\begin{equation}\label{eq:cov1st}
		\begin{split}
			&\left|\langle \sigma_x; I_{\mathcal{H}(x,\omega)}I_{\{B_m(x) \text{ is pivotal for } \mathcal{H}(x,\omega)\}}\rangle_{\Lambda,\beta,h}\right|\leq 2P_{\Lambda,\beta,h}(B_m(x) \text{ is pivotal for } \mathcal{H}(x,\omega))\\
			&\qquad\leq C(m,\beta,h)\sum_{y\in\Lambda_{(1)}\cap B_{m+1}(x)}P_{\Lambda,\beta,h}(y \text{ is pivotal for } \mathcal{H}(y,\omega^{(y)})),
		\end{split}
	\end{equation}
	where we have used the fact that $\{B_m(x) \text{ is pivotal for } \mathcal{H}(x,\omega)\}=\emptyset$ whenever $\Lambda_{(1)}\cap B_{m+1}(x)=\emptyset$ (the latter implies that $B_m(x)\cap \Lambda \cap \mathbb{S}_1=\emptyset$); we have made some local modifications for the spin configuration in $B_{m+3}(x)$ to create a pivotal point $y$ and the cost for such modifications is just a function of $m,\beta,h$ by the finite energy property for Ising measures.  Here the local modifications for $\omega$ within $B_{3m}(y)\supset B_m(x)$ are necessary since we need to make sure that we are free to use any vertices in $\Lambda_{(1)}\cap B_{m+3}(x)$ to construct a pivotal point $y$.
	
	Before we move to the second term on the RHS of \eqref{eq:covdecom}, we note that for a general event $A\in\mathcal{F}_{\Lambda}$, we have
	\begin{equation*}
		\begin{split}
			&\langle \sigma_x;I_A\rangle_{\Lambda,\beta,h}\\
			=&2P_{\Lambda,\beta,h}(\sigma_x=+1)P_{\Lambda,\beta,h}(\sigma_x=-1)\left[P_{\Lambda,\beta,h}(A|\sigma_x=+1)-P_{\Lambda,\beta,h}(A|\sigma_x=-1)\right].
		\end{split}
	\end{equation*}
	This, Proposition \ref{prop:sd} and the mean value theorem imply that
	\begin{equation}\label{eq:cov2nd1}
		\begin{split}
			&\langle \sigma_x; I_{\mathcal{H}(x,\omega)}I_{\{\mathcal{H}(x,\omega) \text{ occurs on } B_m(x)^c\}}\rangle_{\Lambda,\beta,h}=\langle \sigma_x; I_{\{\mathcal{H}(x,\omega) \text{ occurs on } B_m(x)^c\}}\rangle_{\Lambda,\beta,h}\\
			&\qquad\leq 4 e^{-am}\sup_{|t|\leq e^{-am}}\frac{\partial P_{\Lambda,\beta,\mathbf{g}(t)}(\mathcal{H}(x,\omega) \text{ occurs on } B_m(x)^c)}{\partial t},~\forall m\geq M,
		\end{split}
	\end{equation}
	where $\mathbf{g}(t)$ is defined in \eqref{eq:g(t)}.
	Now, Proposition \ref{prop:cov} implies that for any $t$ with $|t|\leq e^{-am}$, we have
	\begin{equation}\label{eq:cov2nd2}
		\begin{split}
			&\frac{\partial P_{\Lambda,\beta,\mathbf{g}(t)}(\mathcal{H}(x,\omega) \text{ occurs on } B_m(x)^c)}{\partial t}=\sum_{y\in \partial B_m(x)\cap \Lambda}\langle \sigma_y; I_{\{\mathcal{H}(x,\omega) \text{ occurs on } B_m(x)^c\}}\rangle_{\Lambda,\beta,\mathbf{g}(t)}\\
			&\qquad \leq \exp\left[4 |\partial B_m| e^{-am}\right] \sum_{y\in \partial B_m(x)\cap \Lambda}\langle \sigma_y; I_{\{\mathcal{H}(x,\omega) \text{ occurs on } B_m(x)^c\}}\rangle_{\Lambda,\beta,h}\\
			&\qquad \leq 2 \sum_{y\in \partial B_m(x)\cap \Lambda}\langle \sigma_y; I_{\{\mathcal{H}(x,\omega) \text{ occurs on } B_m(x)^c\}}\rangle_{\Lambda,\beta,h}, ~ \forall m\geq M,
		\end{split}
	\end{equation}
	where in the last inequality we have made the $M$ from Proposition \ref{prop:sd} larger if necessary such that $\exp\left[4 |\partial B_m| e^{-am}\right] \leq 2$ for each $m\geq M$.
	
	Note that for each $y\in \partial B_m(x)\cap \Lambda$, we have
	\begin{align*}
		&\{\mathcal{H}(x,\omega) \text{ occurs on } B_m(x)^c\}\subset \mathcal{H}(y,\omega),\\
		&\mathcal{H}(y,\omega)\setminus \{\mathcal{H}(x,\omega) \text{ occurs on } B_m(x)^c\}\subset\left\{B_m(x) \text{ is pivotal for } \mathcal{H}(y,\omega)\right\}.
	\end{align*}
	Therefore, we get
	\begin{align*}
		&\langle \sigma_y; I_{\{\mathcal{H}(x,\omega) \text{ occurs on } B_m(x)^c\}}\rangle_{\Lambda,\beta,h}\\
		=&\langle \sigma_y; I_{\mathcal{H}(y,\omega)}\rangle_{\Lambda,\beta,h}-\langle \sigma_y; I_{\{\mathcal{H}(y,\omega)\setminus \{\mathcal{H}(x,\omega) \text{ occurs on } B_m(x)^c\}\}}\rangle_{\Lambda,\beta,h}\\
		\leq&\langle\sigma_y; I_{\mathcal{H}(y,\omega)}\rangle_{\Lambda,\beta,h}+2P_{\Lambda,\beta,h}(B_m(x) \text{ is pivotal for } \mathcal{H}(y,\omega))\\
		\leq &\langle\sigma_y; I_{\mathcal{H}(y,\omega)}\rangle_{\Lambda,\beta,h}+C(m,\beta,h)\sum_{z\in\Lambda_{(1)}\cap B_{m+1}(x)}P_{\Lambda,\beta,h}(z \text{ is pivotal for } \mathcal{H}(z,\omega^{(z)})),
	\end{align*}
	where the last inequality follows as in \eqref{eq:cov1st}. Plugging this into \eqref{eq:cov2nd2} and then into \eqref{eq:cov2nd1}, we obtain that
	\begin{align*}
		&\langle \sigma_x; I_{\mathcal{H}(x,\omega)}I_{\{\mathcal{H}(x,\omega) \text{ occurs on } B_m(x)^c\}}\rangle_{\Lambda,\beta,h}\leq 8e^{-am}\sum_{y\in \partial B_m(x)\cap \Lambda}\langle\sigma_y; I_{\mathcal{H}(y,\omega)}\rangle_{\Lambda,\beta,h}\\
		&\qquad+8e^{-am}C(m,\beta,h)|\partial B_m|\sum_{z\in\Lambda_{(1)}\cap B_{m+1}(x)}P_{\Lambda,\beta,h}(z \text{ is pivotal for } \mathcal{H}(z,\omega^{(z)})).
	\end{align*}
	Combining with \eqref{eq:cov1st} and \eqref{eq:covdecom}, this completes the proof of the lemma.
\end{proof}

We need one more ingredient to prove Proposition \ref{prop:diff}.
\begin{lemma}\label{lem:xytoy}
	Under the assumption of Proposition \ref{prop:diff}, for any $x\in\Lambda$ we have
	\begin{align}
		\begin{split}\label{eq:sumc1}
			&\sum_{\omega\in \{0,1\}^{\Lambda_{(1)}\setminus\{x\}}}p^{|\omega^{-1}(1)|}(1-p)^{|\omega^{-1}(0)|} \langle\sigma_y; I_{\mathcal{H}(y,\omega)}\rangle_{\Lambda,\beta,h}\\
			\leq&\frac{1}{1-p}\sum_{\omega\in \{0,1\}^{\Lambda_{(1)}\setminus\{y\}}}p^{|\omega^{-1}(1)|}(1-p)^{|\omega^{-1}(0)|} \langle\sigma_y; I_{\mathcal{H}(y,\omega)}\rangle_{\Lambda,\beta,h},~\forall y\in\Lambda,
		\end{split}\\
		\begin{split}\label{eq:sumc2}
			&\sum_{\omega\in \{0,1\}^{\Lambda_{(1)}\setminus\{x\}}}p^{|\omega^{-1}(1)|}(1-p)^{|\omega^{-1}(0)|} P_{\Lambda,\beta,h}(y \text{ is pivotal for } \mathcal{H}(y,\omega^{(y)}))\\
			\leq& \frac{1}{1-p}\sum_{\omega\in \{0,1\}^{\Lambda_{(1)}\setminus\{y\}}}p^{|\omega^{-1}(1)|}(1-p)^{|\omega^{-1}(0)|} P_{\Lambda,\beta,h}(y \text{ is pivotal for } \mathcal{H}(y,\omega^{(y)})),~\forall y\in\Lambda_{(1)}.
		\end{split}
	\end{align}
\end{lemma}
\begin{proof}
	For the first inequality, there are four cases.
	\begin{itemize}
		\item $x\notin \Lambda_{(1)}, y\notin\Lambda_{(1)}$. The sums from both sides of \eqref{eq:sumc1} are actually equal.
		\item $x\in \Lambda_{(1)}, y\notin\Lambda_{(1)}$. The one to one correspondence between $\omega\in \{0,1\}^{\Lambda_{(1)}\setminus\{x\}}$ from the left hand side (LHS) of \eqref{eq:sumc1} and $\omega\in \{0,1\}^{\Lambda_{(1)}\setminus\{y\}}$ with $\omega_x=0$ from the RHS of \eqref{eq:sumc1} while keeping all other $\omega_z$ unchanged satisfies that the ratio of of the corresponding terms is exactly $1/(1-p)$.
		\item $x\in \Lambda_{(1)}, y\in\Lambda_{(1)}$ and $y\neq x$. The one to one correspondence between $\omega\in \{0,1\}^{\Lambda_{(1)}\setminus\{x\}}$ with $\omega_y=0$ from the LHS of \eqref{eq:sumc1} and $\omega\in \{0,1\}^{\Lambda_{(1)}\setminus\{y\}}$ with $\omega_x=0$ from the RHS of \eqref{eq:sumc1}  while keeping all other $\omega_z$ unchanged satisfies that the ratio of the corresponding terms is always $1$. The one to one correspondence between $\omega\in \{0,1\}^{\Lambda_{(1)}\setminus\{x\}}$ with $\omega_y=1$ from the LHS of \eqref{eq:sumc1} and $\omega\in \{0,1\}^{\Lambda_{(1)}\setminus\{y\}}$ with $\omega_x=0$ from the RHS of \eqref{eq:sumc1} while keeping all other $\omega_z$ unchanged satisfies that the ratio of of the corresponding terms is always $\frac{p}{1-p}$. 
		\item $x\notin \Lambda_{(1)}, y\in\Lambda_{(1)}$. The sum on the LHS of \eqref{eq:sumc1} with the restriction $\omega_y=0$ equals that $(1-p)$ times the sum on the RHS of \eqref{eq:sumc1}. The sum of the LHS of \eqref{eq:sumc1} with the restriction $\omega_y=1$ equals that $p$ times the sum on the RHS of \eqref{eq:sumc1} 
	\end{itemize}
	This completes the proof of \eqref{eq:sumc1}. The proof of \eqref{eq:sumc2} is similar and simpler.
\end{proof}
We are ready to prove Proposition \ref{prop:diff}.
\begin{proof}[Proof of Proposition \ref{prop:diff}]
	By Lemma \ref{lem:dirivative}, we have
	\begin{align*}
		\frac{\partial\mu_{\Lambda,p,h}(\mathcal{H})}{\partial h}&=\sum_{x\in \Lambda}\sum_{\omega\in \{0,1\}^{\Lambda_{(1)}}}p^{|\omega^{-1}(1)|}(1-p)^{|\omega^{-1}(0)|}\langle\sigma_x; I_{\{\mathcal{H} \text{ occurs on } \mathbb{S}_0 \cup \omega^{-1}(1)\}}\rangle_{\Lambda,\beta,h}\\
		&=\sum_{x\in \Lambda}\sum_{\omega\in \{0,1\}^{\Lambda_{(1)}\setminus\{x\}}}p^{|\omega^{-1}(1)|}(1-p)^{|\omega^{-1}(0)|} \big[p\langle\sigma_x; I_{\mathcal{H}(x,\omega)}\rangle_{\Lambda,\beta,h}\\
		&\qquad +(1-p)\langle\sigma_x; I_{\{\mathcal{H} \text{ occurs on } \mathbb{S}_0 \cup \omega^{-1}(1)\}}\rangle_{\Lambda,\beta,h}\big].
	\end{align*}
	Note that
	\begin{align*}
		&p\langle\sigma_x; I_{\mathcal{H}(x,\omega)}\rangle_{\Lambda,\beta,h} +(1-p)\langle\sigma_x; I_{\{\mathcal{H} \text{ occurs on } \mathbb{S}_0 \cup \omega^{-1}(1)\}}\rangle_{\Lambda,\beta,h}\\
		=&\langle\sigma_x; I_{\mathcal{H}(x,\omega)}\rangle_{\Lambda,\beta,h}-(1-p)\left[\langle\sigma_x; I_{\mathcal{H}(x,\omega)}\rangle_{\Lambda,\beta,h}-\langle\sigma_x; I_{\{\mathcal{H} \text{ occurs on } \mathbb{S}_0 \cup \omega^{-1}(1)\}}\rangle_{\Lambda,\beta,h}\right],
	\end{align*}
	where by \eqref{eq:+piv}, we have
	\begin{equation*}
		\left|\langle\sigma_x; I_{\mathcal{H}(x,\omega)}\rangle_{\Lambda,\beta,h}-\langle\sigma_x; I_{\{\mathcal{H} \text{ occurs on } \mathbb{S}_0 \cup \omega^{-1}(1)\}}\rangle_{\Lambda,\beta,h}\right|\leq 2 P_{\Lambda,\beta,h}(x \text{ is}+\text{pivotal for } \mathcal{H}(x,\omega)).
	\end{equation*}
	Note that the LHS of the last displayed inequality is $0$ if $x\notin\Lambda_{(1)}$. Also note that the event $\{x \text{ is pivotal for } \mathcal{H}(x,\omega)\}$ does not depend on  the status of $\sigma_x$; so there exist constants $c_1=c_1(\beta,h), C_1=C_1(\beta,h)\in(0,\infty)$ such that
	\[c_1\leq\frac{P_{\Lambda,\beta,h}(\sigma_x=+1|x \text{ is pivotal for } \mathcal{H}(x,\omega))}{P_{\Lambda,\beta,h}(\sigma_x=-1|x \text{ is pivotal for } \mathcal{H}(x,\omega))}\leq C_1.\]
	Therefore, Proposition \ref{prop:diff} follows if we can prove that there exists a continuous function $\gamma$ such that
	\begin{equation}\label{eq:covtopiv}
		\begin{split}
			&\sum_{x\in \Lambda}\sum_{\omega\in \{0,1\}^{\Lambda_{(1)}\setminus\{x\}}}p^{|\omega^{-1}(1)|}(1-p)^{|\omega^{-1}(0)|} \langle\sigma_x; I_{\mathcal{H}(x,\omega)}\rangle_{\Lambda,\beta,h}\\
			\leq& \gamma(\beta,h)\sum_{x\in \Lambda_{(1)}}\sum_{\omega\in \{0,1\}^{\Lambda_{(1)}\setminus\{x\}}}p^{|\omega^{-1}(1)|}(1-p)^{|\omega^{-1}(0)|} P_{\Lambda,\beta,h}(x \text{ is pivotal for } \mathcal{H}(x,\omega)).
		\end{split}
	\end{equation}
	Let 
	\[\mathbb{P}_{p}(\omega):=p^{|\omega^{-1}(1)|}(1-p)^{|\omega^{-1}(0)|},~\forall \omega\in \{0,1\}^{\Lambda_{(1)}\setminus\{x\}}.\]
	Then Lemma \ref{lem:covtopiv} implies that for each $m\geq M$, we have
	\begin{equation}\label{eq:sumcov1}
		\begin{split}
			&\sum_{x\in \Lambda}\sum_{\omega\in \{0,1\}^{\Lambda_{(1)}\setminus\{x\}}}p^{|\omega^{-1}(1)|}(1-p)^{|\omega^{-1}(0)|} \langle\sigma_x; I_{\mathcal{H}(x,\omega)}\rangle_{\Lambda,\beta,h}\\
			\leq&C(m,\beta,h)\sum_{x\in \Lambda}\sum_{\omega\in \{0,1\}^{\Lambda_{(1)}\setminus\{x\}}}\mathbb{P}_{p}(\omega)\sum_{y\in\Lambda_{(1)}\cap B_{m+1}(x)}P_{\Lambda,\beta,h}(y \text{ is pivotal for } \mathcal{H}(y,\omega^{(y)}))\\
			&\qquad +8e^{-ma}\sum_{x\in \Lambda}\sum_{\omega\in \{0,1\}^{\Lambda_{(1)}\setminus\{x\}}}\mathbb{P}_{p}(\omega)\sum_{y\in\partial B_m(x)\cap \Lambda}\langle \sigma_y; I_{\mathcal{H}(y,\omega)}\rangle_{\Lambda,\beta,h}.
		\end{split}
	\end{equation}
	Applying Lemma \ref{lem:xytoy}, we obtain that the LHS of \eqref{eq:sumcov1}
	\begin{equation}\label{eq:sumcov2}
		\begin{split}
			\leq&\frac{C(m,\beta,h)}{1-p}\sum_{x\in \Lambda}\sum_{y\in\Lambda_{(1)}\cap B_{m+1}(x)}\sum_{\omega\in \{0,1\}^{\Lambda_{(1)}\setminus\{y\}}}\mathbb{P}_{p}(\omega)P_{\Lambda,\beta,h}(y \text{ is pivotal for } \mathcal{H}(y,\omega^{(y)}))\\
			&\qquad +\frac{8e^{-ma}}{1-p}\sum_{x\in \Lambda}\sum_{y\in\partial B_m(x)\cap \Lambda}\sum_{\omega\in \{0,1\}^{\Lambda_{(1)}\setminus\{y\}}}\mathbb{P}_{p}(\omega)\langle \sigma_y; I_{\mathcal{H}(y,\omega)}\rangle_{\Lambda,\beta,h}\\
			\leq&\frac{C(m,\beta,h)}{1-p}|B_{m+1}|\sum_{y\in\Lambda_{(1)}}\sum_{\omega\in \{0,1\}^{\Lambda_{(1)}\setminus\{y\}}}\mathbb{P}_{p}(\omega)P_{\Lambda,\beta,h}(y \text{ is pivotal for } \mathcal{H}(y,\omega^{(y)}))\\
			&\qquad +\frac{8e^{-ma}}{1-p}|\partial B_m|\sum_{y\in\Lambda}\sum_{\omega\in \{0,1\}^{\Lambda_{(1)}\setminus\{y\}}}\mathbb{P}_{p}(\omega)\langle \sigma_y; I_{\mathcal{H}(y,\omega)}\rangle_{\Lambda,\beta,h}.
		\end{split}
	\end{equation}
	Choose $N=N(\beta,p) \geq M$ such that
	\[\frac{8e^{-ma}}{1-p}|\partial B_m|<\frac{1}{2},~\forall m\geq N.\]
	Then for each $m\geq N$,  \eqref{eq:sumcov1} and \eqref{eq:sumcov2} give that
	\begin{align*}
		&\sum_{x\in \Lambda}\sum_{\omega\in \{0,1\}^{\Lambda_{(1)}\setminus\{x\}}}\mathbb{P}_{p}(\omega) \langle\sigma_x; I_{\mathcal{H}(x,\omega)}\rangle_{\Lambda,\beta,h}\\ \leq&\frac{C(m,\beta,h)}{1-p}|B_{m+1}|\sum_{y\in\Lambda_{(1)}}\sum_{\omega\in \{0,1\}^{\Lambda_{(1)}\setminus\{y\}}}\mathbb{P}_{p}(\omega)P_{\Lambda,\beta,h}(y \text{ is pivotal for } \mathcal{H}(y,\omega^{(y)}))\\
		&\quad +\frac{1}{2}\sum_{y\in\Lambda}\sum_{\omega\in \{0,1\}^{\Lambda_{(1)}\setminus\{y\}}}\mathbb{P}_{p}(\omega)\langle \sigma_y; I_{\mathcal{H}(y,\omega)}\rangle_{\Lambda,\beta,h}
	\end{align*}
	where the second double sum on the RHS of the inequality is the same as the first double sum on the LHS. So we obtain
	that for each $m\geq N$, we have
	\begin{align*}
		&\sum_{x\in \Lambda}\sum_{\omega\in \{0,1\}^{\Lambda_{(1)}\setminus\{x\}}}\mathbb{P}_{p}(\omega) \langle\sigma_x; I_{\mathcal{H}(x,\omega)}\rangle_{\Lambda,\beta,h}\\
		\leq&\frac{2C(m,\beta,h)}{1-p}|B_{m+1}|\sum_{y\in\Lambda_{(1)}}\sum_{\omega\in \{0,1\}^{\Lambda_{(1)}\setminus\{y\}}}\mathbb{P}_{p}(\omega)P_{\Lambda,\beta,h}(y \text{ is pivotal for } \mathcal{H}(y,\omega^{(y)}))\\
		\leq &\frac{2C(m,\beta,h)}{1-p}|B_{m+1}|\sum_{y\in\Lambda_{(1)}}\frac{2^{|B_{3m}|}}{(\min\{p,1-p\})^{|B_{3m}|}}\\
		&\qquad\times\sum_{\omega\in \{0,1\}^{\Lambda_{(1)}\setminus\{y\}}}\mathbb{P}_{p}(\omega)P_{\Lambda,\beta,h}(y \text{ is pivotal for } \mathcal{H}(y,\omega)),
	\end{align*}
	where the last inequality follows because of the finite energy property for independent percolation and that  the map from $\omega$ to $\omega^{(y)}$ is at most $|B_{3m}|$ to $1$. This completes the proof of \eqref{eq:covtopiv} and thus the proposition.
\end{proof}

\section{Proof of the main result}\label{sec:pf}
In this section, we use the differential inequality that we derived in Section \ref{sec:diff} to prove our main result.
We always fix $\beta\in[0,\beta_c(\mathbb{T}\times\mathbb{Z}))$ in this section. We first show that the infinite-volume Ising measure $P_{\beta,h}$ has a sharp phase transition in the $h$ variable. Let $\{0\overset{+}{\longleftrightarrow}\partial B_n\}$ denote the event that there is a path of $+$ spins which connects the origin to $\partial B_n$, and let $\{0\overset{+}{\longleftrightarrow}\infty\}$ denote the event that the origin belongs to an infinite $+$ cluster.
\begin{proposition}[\cite{DCRT19}]\label{prop:sharp}
	For each $\beta\in[0,\beta_c(\mathbb{T}\times\mathbb{Z}))$, there exists $h_c=h_c(\beta)\in\mathbb{R}$ such that
	\begin{itemize}
		\item there exists $c>0$ such that $P_{\beta,h}(0\overset{+}{\longleftrightarrow}\infty)\geq c(h-h_c)$ for each $h>h_c$ which is close to $h_c$;
		\item for each $h<h_c$, there exists $c_1=c_1(h)\in(0,\infty)$ such that
		\[P_{B_n,\beta,h}^{+}(0\overset{+}{\longleftrightarrow}\partial B_n)\leq \exp[-c_1n].\]
	\end{itemize}
\end{proposition}
\begin{proof}
	One can check that the proof of Theorem 1.2 in \cite{DCRT19} also applies to the measure $P_{\beta,h}$. Indeed,  the two key ingredients of the proof are valid in our setting: the FKG lattice property for the measure and a derivative formula similar to (16) in \cite{DCRT19}. We also note that the argument applies to all $\beta\in[0,\infty)$, although we only need it for $\beta<\beta_c(\mathbb{T}\times\mathbb{Z})$.
\end{proof}

\begin{proof}[Proof of Theorem \ref{thm1}]
	Since the $\gamma(p,h)$ from Proposition \ref{prop:diff} is continuous in $(p,h)\in(0,1)\times\mathbb{R}$, we can find $\delta\in(0,1/4)$, $H>2\delta$ and $\Gamma>0$ such that
	\[\gamma(p,h)\leq \Gamma,~\forall (p,h)\in[\delta,1-\delta]\times[-H,H].\]
	Take $\theta\in(0,\pi/2)$ such that $\tan\theta=1/\Gamma$.
	Suppose that $(p_0,h_0)\in[2\delta,1-2\delta]\times[-H/2,H/2]$ and define
	\[(p(t),h(t)):=(p_0,h_0)+t\delta(\cos\theta,-\sin\theta),~t\in[0,1].\]
	Then it is clear that $(p(t),h(t))\in[\delta,1-\delta]\times[-H,H]$ for each $t\in[0,1]$. Let $\Lambda\subset\mathbb{Z}^3$ be finite and $\mathcal{H}$ be the event defined in Proposition \ref{prop:diff}. Then Proposition \ref{prop:diff} implies that
	\begin{align*}
		\frac{d\mu_{\Lambda,p(t),h(t)}(\mathcal{H})}{dt}&=\delta \cos\theta \left.\frac{\partial\mu_{\Lambda,p,h}(\mathcal{H})}{\partial p}\right|_{(p,h)=(p(t),h(t))}-\delta \sin\theta \left.\frac{\partial\mu_{\Lambda,p,h}(\mathcal{H})}{\partial h}\right|_{(p,h)=(p(t),h(t))}\\
		&\geq\delta [\cos\theta-\Gamma \sin\theta] \left.\frac{\partial\mu_{\Lambda,p,h}(\mathcal{H})}{\partial p}\right|_{(p,h)=(p(t),h(t))}=0,~\forall t\in[0,1].
	\end{align*}
	Integrating the last inequality from $0$ to $1$ gives
	\[\mu_{\Lambda,p(1),h(1)}(\mathcal{H})\geq\mu_{\Lambda,p_0,h_0}(\mathcal{H}).\]
	Setting $\Lambda\uparrow\mathbb{Z}^3$, we get
	\begin{equation}\label{eq:probA}
		\mu_{\mathbb{Z}^3,p(1),h(1)}(\mathcal{H})\geq\mu_{\mathbb{Z}^3,p_0,h_0}(\mathcal{H}),
	\end{equation}
	where $\mu_{\mathbb{Z}^3,p,h}(\mathcal{H})$ is defined by
	\[\mu_{\mathbb{Z}^3,p,h}(\mathcal{H}):=\int_{\{0,1\}^{\mathbb{S}_1\setminus\mathbb{S}_0}}P_{\beta,h}(\mathcal{H} \text{ occurs on }\mathbb{S}_0\cup\omega^{-1}(1)) \mathbb{P}_{p}(d\omega)\]
	with $\mathbb{P}_p$ denoting the independent Bernoulli site percolation measure with parameter $p$.
	We now pick $(p_0,h_0)=(1/2,0)$. Then the inequality \eqref{eq:probA} and the monotonicity of $\mu_{\mathbb{Z}^3,p,h}(\mathcal{H})$ in $p$ imply that
	\[\mu_{\mathbb{Z}^3,1,-\delta\sin\theta}(\mathcal{H})\geq\mu_{\mathbb{Z}^3,1/2+\delta\cos\theta,-\delta\sin\theta}(\mathcal{H})\geq\mu_{\mathbb{Z}^3,1/2,0}(\mathcal{H})\geq \mu_{\mathbb{Z}^3,0,0}(\mathcal{H}).\]
	See Figure \ref{fig:Mon} for an illustration of this monotonicity.
	Recall that $\mathcal{H}:=\{\exists \text{ LR} + \text{crossing in }B_n\cap \mathbb{S}_1 \}$. The structure of $\mathbb{T}$ implies that the complement of $\{\exists \text{ LR} + \text{crossing in }B_n\cap \mathbb{S}_0\}$ is $\{\exists \text{ top to bottom } - \text{crossing in }B_n\cap \mathbb{S}_0 \}$. This, combined with the $\pm$ symmetry of $P_{\beta,0}$, gives that
	\[\mu_{\mathbb{Z}^3,0,0}(\mathcal{H})=P_{\beta,0}(\exists \text{ LR} + \text{crossing in }B_n\cap \mathbb{S}_0)=1/2,~\forall n\in\mathbb{N}.\]
	So we just proved
	\[P_{\beta,-\delta\sin\theta}(\exists \text{ LR} + \text{crossing in }B_n\cap \mathbb{S}_1 )=\mu_{\mathbb{Z}^3,1,-\delta\sin\theta}(\mathcal{H})\geq 1/2,~\forall n\in\mathbb{N}.\]
	Proposition \ref{prop:sharp} then implies that
	\[h_c(\beta)\leq -\delta\sin\theta<0,\]
	and there is an infinite $+$cluster $P_{\beta,h}$ almost surely whenever $h>h_c(\beta)$; this is because if $h_c(\beta)> -\delta\sin\theta$, then we would have
	\begin{align*}
		P_{\beta,-\delta\sin\theta}(\exists \text{ LR} + \text{crossing in }B_n\cap \mathbb{S}_1 )&\leq P_{\beta,-\delta\sin\theta}(\exists x\in B_n  \text{ s.t. } x\overset{+}{\longleftrightarrow} (x+\partial B_n) )\\
		&\leq |B_n| P^{+}_{B_n,\beta,-\delta\sin\theta}(0\overset{+}{\longleftrightarrow} \partial B_n)\rightarrow 0 \text{ as } n\rightarrow \infty.
	\end{align*}
	By $\pm$ symmetry, there is also an infinite $-$cluster $P_{\beta,h}$ almost surely whenever $h<|h_c(\beta)|$. The uniqueness part follows from the Burton-Keane argument \cite{BK89}.
	
	\begin{figure}
		\begin{center}
			\includegraphics{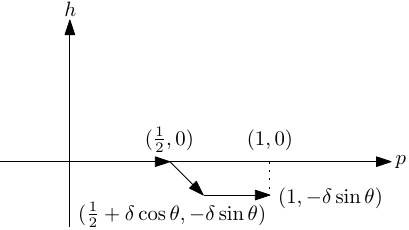}
			\caption{An illustration of the monotonicity of $\mu_{\mathbb{Z}^3,p,h}(\mathcal{H})$ with respect to $p$ and $h$. The line segment from $(1/2,0)$ to $(1/2+\delta \cos \theta, -\delta \sin \theta)$ (with slope $-\tan \theta$) represents the curve $\{(p(t),h(t)): 0\leq t\leq 1\}$ with $(p_0,h_0)=(1/2,0)$. $\mu_{\mathbb{Z}^3,p,h}(\mathcal{H})$ is increasing along the three consecutive directed line segments started from $(0,0)$ and ended at $(1,-\delta \sin\theta)$. }\label{fig:Mon}
		\end{center}
	\end{figure}
\end{proof}

\begin{appendix}
\section*{A reversed inequality}\label{appn} 
In this appendix, we prove that the reversed inequality (with a different function $\gamma$) in Proposition~\ref{prop:diff} also holds for general increasing events $A\in\mathcal{F}_{\Lambda}$. We do not need this for the proof of our main result, but it might be useful elsewhere. Recall the definition of $\sigma^{\Delta\pm}$ in \eqref{eq:sigmamod}; if $\Delta=\{x\}$, we simply write $\sigma^{x\pm}$. For any increasing event $A\in\mathcal{F}_{\Lambda}$, we define
\[A_x^{+}:=\{\sigma\in\{-1,+1\}^{\Lambda}: \sigma^{x+}\in A\},~A_x^{-}:=\{\sigma\in\{-1,+1\}^{\Lambda}: \sigma^{x-}\in A\}.\]
Note that $A_x^{\pm}$ depend on $x\in\Lambda$ but not on the status of $\sigma_x$.
Then for any $x\in\Lambda$, we have
\begin{gather*}
	\{x \text{ is pivotal for }A\}=A_x^+\setminus A_x^-,\{x \text{ is} + \text{pivotal for }A\}=A\setminus A_x^-,\\
	\{x \text{ is} - \text{pivotal for }A\}=A_x^+\setminus A.
\end{gather*}
\begin{lemma}\label{lem:rev}
	For any $x\in\Lambda$ and any increasing event $A\in\mathcal{F}_{\Lambda}$, we have
	\begin{itemize}
		\item if $A_x^-=\emptyset$ and $P_{\Lambda,\beta,h}(x \text{ is pivotal for }A)\neq0$, then
		\[P_{\Lambda,\beta,h}(x \text{ is pivotal for }A)=\frac{\langle \sigma_x; I_A\rangle_{\Lambda,\beta,h}}{2P_{\Lambda,\beta,h}(\sigma_x=-1)}\frac{1}{P_{\Lambda,\beta,h}(\sigma_x=+1|x \text{ is pivotal for }A)};\]
		\item if $A_x^-\neq\emptyset$, then
		\begin{align*}
			P_{\Lambda,\beta,h}(x \text{ is pivotal for }A)&=\frac{\langle \sigma_x; I_A\rangle_{\Lambda,\beta,h}}{2P_{\Lambda,\beta,h}(\sigma_x=+1|A_x^+)}+\frac{\langle \sigma_x; I_A\rangle_{\Lambda,\beta,h}}{2P_{\Lambda,\beta,h}(\sigma_x=-1|A_x^-)}\\
			&\qquad-\frac{\langle \sigma_x; I_{A_x^+}\rangle_{\Lambda,\beta,h}}{2P_{\Lambda,\beta,h}(\sigma_x=+1|A)}-\frac{\langle \sigma_x; I_{A_x^-}\rangle_{\Lambda,\beta,h}}{2P_{\Lambda,\beta,h}(\sigma_x=-1|A)}.
		\end{align*}
	\end{itemize}
\end{lemma}
\begin{remark}
	Note that if $A_x^-\neq\emptyset$ then $\{A,\sigma_x=-1\}=\{A_x^{-},\sigma_x=-1\}\neq \emptyset$. Since $\{x \text{ is pivotal for }A\}$ and $A_x^{\pm}$ do not depend on the status of $\sigma_x$, all denominators in the lemma except the last two are between $c_1(\beta,h)$ and $c_2(\beta,h)$ with
	\begin{align*}
		&c_1(\beta,h):=\min\{P_{\Lambda,\beta,h}(\sigma_x=+1|\sigma_y=-1,\forall y\neq x), P_{\Lambda,\beta,h}(\sigma_x=-1|\sigma_y=+1,\forall y\neq x)\},\\
		&c_2(\beta,h):=\max\{P_{\Lambda,\beta,h}(\sigma_x=+1|\sigma_y=+1,\forall y\neq x), P_{\Lambda,\beta,h}(\sigma_x=-1|\sigma_y=-1,\forall y\neq x)\}.
	\end{align*}
	So we have pointwise (in $x\in\Lambda$) upper bound between $P_{\Lambda,\beta,h}(x \text{ is pivotal for }A)$ and $\langle \sigma_x; I_A\rangle_{\Lambda,\beta,h}$. But by considering the special event $A=I_{\{\sigma_0=+1\}}$, one can see that a pointwise lower bound doesn't hold for general $A$ (it fails for each $x\neq0$). This suggests that the other direction of the inequality when summing over $x\in\Lambda$, as in Proposition \ref{prop:diff}, is much harder to obtain.
\end{remark}
For ease of reference, we state the following corollary, which corresponds to the reverse inequality in Proposition \ref{prop:diff}.
\begin{corollary}
	Under the assumption of Proposition \ref{prop:diff}, with $\gamma$ replaced by $\tilde{\gamma}$, 
	\[\frac{\partial \mu_{\Lambda,p,h}(\mathcal{H})}{\partial h}\geq \tilde{\gamma}(p,h) \frac{\partial \mu_{\Lambda,p,h}(\mathcal{H})}{\partial p}.\]
	The same inequality also holds for all increasing events $A\in\mathcal{F}_{\Lambda}$. 
\end{corollary}
\begin{proof}[Proof of Lemma \ref{lem:rev}]
	If $A_x^-=\emptyset$, then $\{A \text{ occurs on }\Lambda\setminus\{x\}\}=\emptyset$. So by \eqref{eq:pivcom}, we have
	\begin{equation}\label{eq:covtoA}
		\begin{split}
			\langle \sigma_x; I_A\rangle_{\Lambda,\beta,h}&=\langle \sigma_x; I_AI_{\{x \text{ is pivotal for }A\}}\rangle_{\Lambda,\beta,h}=\langle \sigma_xI_AI_{\{x \text{ is pivotal for }A\}}\rangle_{\Lambda,\beta,h}\\
			&\qquad-\langle \sigma_x\rangle_{\Lambda,\beta,h} \langle I_AI_{\{x \text{ is pivotal for }A\}}\rangle_{\Lambda,\beta,h}\\
			&=\left[1-\langle \sigma_x\rangle_{\Lambda,\beta,h}\right]P_{\Lambda,\beta,h}(A \cap \{x \text{ is pivotal for }A\})\\
			&=2P_{\Lambda,\beta,h}(\sigma_x=-1)P_{\Lambda,\beta,h}(A),
		\end{split}
	\end{equation}
	where we have used $A\subset A_x^+=\{x \text{ is pivotal for }A\}$ in the last inequality. On the other hand, we also have
	\begin{align*}
		&P_{\Lambda,\beta,h}(x \text{ is pivotal for }A)=P_{\Lambda,\beta,h}(x \text{ is pivotal for }A, \sigma_x=+1)\\
		&\qquad+P_{\Lambda,\beta,h}(x \text{ is pivotal for }A, \sigma_x=-1)\\
		=&\left[1+\frac{P_{\Lambda,\beta,h}(\sigma_x=-1|x \text{ is pivotal for }A)}{P_{\Lambda,\beta,h}(\sigma_x=+1|x \text{ is pivotal for }A)}\right]P_{\Lambda,\beta,h}(x \text{ is pivotal for }A, \sigma_x=+1)\\
		=&\frac{P_{\Lambda,\beta,h}(A_x^+, \sigma_x=+1)}{P_{\Lambda,\beta,h}(\sigma_x=+1|x \text{ is pivotal for }A)}=\frac{P_{\Lambda,\beta,h}(A)}{P_{\Lambda,\beta,h}(\sigma_x=+1|x \text{ is pivotal for }A)}.
	\end{align*}
	This and \eqref{eq:covtoA} completes the proof of the first identity in the lemma.
	
	For the second identity, we have
	\begin{align*}
		&P_{\Lambda,\beta,h}(x \text{ is pivotal for }A)=P_{\Lambda,\beta,h}(A_x^+)-P_{\Lambda,\beta,h}(A_x^-)\\
		=&\frac{P_{\Lambda,\beta,h}(A,\sigma_x=+1)}{P_{\Lambda,\beta,h}(\sigma_x=+1|A_x^+)}-\frac{P_{\Lambda,\beta,h}(A,\sigma_x=-1)}{P_{\Lambda,\beta,h}(\sigma_x=-1|A_x^-)}\\
		=&\frac{P_{\Lambda,\beta,h}(A,\sigma_x=+1)-P_{\Lambda,\beta,h}(A)P_{\Lambda,\beta,h}(\sigma_x=+1)}{P_{\Lambda,\beta,h}(\sigma_x=+1|A_x^+)}\\
		&\qquad+\frac{P_{\Lambda,\beta,h}(A)P_{\Lambda,\beta,h}(\sigma_x=-1)-P_{\Lambda,\beta,h}(A,\sigma_x=-1)}{P_{\Lambda,\beta,h}(\sigma_x=-1|A_x^-)}\\
		&\qquad+P_{\Lambda,\beta,h}(A)\frac{P_{\Lambda,\beta,h}(A_x^+)P_{\Lambda,\beta,h}(\sigma_x=+1)-P_{\Lambda,\beta,h}(A_x^+,\sigma_x=+1)}{P_{\Lambda,\beta,h}(A,\sigma_x=+1)}\\
		&\qquad+P_{\Lambda,\beta,h}(A)\frac{P_{\Lambda,\beta,h}(A_x^-,\sigma_x=-1)-P_{\Lambda,\beta,h}(A_x^-)P_{\Lambda,\beta,h}(\sigma_x=-1)}{P_{\Lambda,\beta,h}(A,\sigma_x=-1)},
	\end{align*}
	which clearly implies the desired identity.
\end{proof}
\end{appendix}

\begin{acks}[Acknowledgments]
The authors thank Federico Camia, Frederik Klausen and Chuck Newman for useful comments on an early draft. The authors also thank the reviewers for useful comments and suggestions.
\end{acks}

\begin{funding}
This research was partially supported by NSFC Tianyuan Key Program Project (No.12226001) and NSFC General Program (No. 12271284).
\end{funding}

\bibliographystyle{imsart-number} 
\bibliography{reference}       

\end{document}